\documentclass[12pt,oneside]{article}

\usepackage{mathrsfs,amsfonts,amsmath,amssymb}
\usepackage{latexsym}
\usepackage{comment} 
\usepackage{graphicx}
\usepackage[T1]{fontenc}
\usepackage[utf8]{inputenc}
\usepackage{authblk}
\usepackage{caption}

\newcommand{\qed}{{} \hfill \mbox{$\Box$}}

\newtheorem{satz}{Theorem}

\newtheorem{theorem}[satz]{Theorem}
\newtheorem{lemma}[satz]{Lemma}
\newtheorem{definition}[satz]{Definition}

\newtheorem{remark}[satz]{Remark}

\newtheorem{conjecture}[satz]{Conjecture}
\newtheorem{claim}[satz]{Claim}

\def\F{\mathbb {F}}

\def\d{\delta}

\def\({\big (}
\def\){\big )}

\def\ge{\geqslant}
\def\_phi{\varphi}

 \makeatletter
      \def\@setcopyright{}
      \def\serieslogo@{}
      \makeatother

\begin{document}

\title{Tilted Corners in Integer Grids}

\author[1]{I.~D.~Shkredov}

\author[2]{J.~ Solymosi}

\affil[1]{Steklov Mathematical Institute, Moscow \& IITP, Moscow \& MIPT, Dolgoprudnii}
\affil[2]{University of British Columbia, Vancouver}


   \date{}

\maketitle

{\centering Dedicated to the memory of Ron Graham\par}

\begin{abstract}
It was proved by Ron Graham and the second author that for any coloring of the $N \times N$ grid using fewer than $\log \log N$ colours,
one can always find a monochromatic isosceles  right triangle, a triangle with
vertex coordinates $(x, y),(x + d, y),$ and $(x, y + d).$ In this paper we are asking questions where not only axis-parallel, but tilted  isosceles right triangles are considered as well. Both colouring and density variants of the problem will be discussed.
\end{abstract}

\section{Introduction}

In this paper we are going to consider several problems inspired by questions raised by Ron Graham. After learning Szemer\'edi's proof of the Erd\H{o}s-Tur\'an conjecture on $4$-term arithmetic progressions in dense subsets of integers \cite{Sze}, Graham asked the following question: Is it true that for any real number $\delta > 0$ there is a natural number $N_0 = N_0(\delta)$ such that for $N>N_0$ every subset of $[N]\times [N]$ of size at least $\delta N^2$ contains a square, i.e., a quadruple of the form $\{(a, b),(a + d, b),(a, b + d),(a + d, b + d)\}$  for some integer $d \neq 0$ ? $([N] = \{0, 1, 2,\ldots,N - 1\}).$
Using the full power of Szemer\'edi's theorem on $k$-term arithmetic progressions, Ajtai and
Szemer\'edi in \cite{ASz} proved a simpler statement: for sufficiently large $N,$ every subset of $[N]\times[N]$
of size at least $\delta N^2$ contains corners, three points with coordinates $\{(a, b),(a + d, b),(a, b + d)\}
\footnote{Through the paper we are assuming that the corners and squares are not degenerate, $d \neq 0.$}$ (see also in \cite{VV})
 Later F\"urstenberg and Katznelson proved a much stronger,
general theorem in \cite{FK}, but their proof didn't give an explicit bound
as it uses ergodic theory. After Tim Gowers gave an analytical proof for Szemer\'edi's theorem (receiving a \$1,000 check from Ron Graham who paid rewards offered by Paul Erd\H{o}s) he again raised the question of finding a quantitative proof for Graham's question. Such proof was given by the second author in \cite{Soly} using a hypergraph regularity lemma of Frankl and R\"odl \cite{FR}. Although it is quantitative, it is still very far from a conjecture of Graham:

\begin{conjecture}[Ron Graham \cite{Gr}]
Given a set of lattice points in the plane
$$ S = \{p_1, p_2,\ldots,p_i, p_{i+1},\ldots\},$$
let us denote the distance of $p_i$ from the origin by $d_i$. If

\[
\sum_{i=1}^\infty\frac{1}{d_i^2}=\infty \,,
\]
then $S$ contains the four vertices of an axes-parallel square.
\end{conjecture}
The second author of this paper heard the conjecture from Ron Graham multiple times, with increasing reward offer. Once Ron said {\em ``I think it is a safe bet to offer \$1,000 for the solution. I don't think I ever have to pay that.''}

Even after the recent breakthrough of Bloom and Sisask, breaking the logarithmic barrier in Roth's theorem on three term arithmetic progressions \cite{BS_AP3}, we are very far from such bounds. We offer a weaker conjecture, changing squares to corners even allowing rotated (tilted) corners. In light of Theorem \ref{Corner} below, it might be accessible using techniques available now.

\begin{conjecture}
Given a set of lattice points in the plane
$$ S = \{p_1, p_2,\ldots,p_i, p_{i+1},\ldots\},$$
let us denote the distance of $p_i$ from the origin by $d_i$. If

\[
\sum_{i=1}^\infty\frac{1}{d_i^2}=\infty \,,
\]
then $S$ contains the three vertices of an isosceles right triangle.
\end{conjecture}

\noindent
If we restrict our attention to axis parallel corners, then the best known density bound for the Ajtai-Szemer\'edi theorem belongs to the first author: 
\begin{theorem}[Shkredov \cite{s_IAN}] 
For sufficiently large $N,$ every subset of $[N]\times[N]$
of size at least $N^2/(\log\log{N})^C$ contains corners, three points with coordinates $$\{(a, b),(a + d, b),(a, b + d)\}.$$
\end{theorem}

This problem is one of the few examples where the colouring variant has a better (known) bound than its density version.

\begin{theorem}[Graham-Solymosi \cite{GS}]\label{GrSo}
For $N$ large enough, any colouring of the $N \times N$ grid using fewer than $\log \log N$ colours,
one can always find a monochromatic isosceles right triangle, a triangle with
vertex coordinates $(x, y),(x + d, y),$ and $(x, y + d).$    
\end{theorem}

In what follows we will see variants of the above mentioned problems. The next section is about saturated point sets of the integer grid, sets without corners (or squares) which are maximal, adding any further gridpoint will result a corner (or square). 

In Section \ref{dens} we summarize what are the best density results one can expect using the available techniques. Unfortunately we can't provide full proofs here, they are quite technical, but the arguments are hopefully complete enough that experts could reconstruct the proofs.

The last section is about related colouring problems, briefly addressing Euclidean Ramsey type problems, one of the many fields where Ron Graham has made significant impact. We close this introduction with a nice result of Ron, similar to problems we are going to consider in this paper, finding monochromatic right triangles in integer grids.

\begin{theorem}[Graham \cite{Gr2}]
    For any $r,$ there exists a positive integer $T(r)$ so that in any
$r$-coloring of the lattice points $\mathbb{Z}^2$ of the plane, there is always a monochromatic
right triangle with area exactly $T(r).$
\end{theorem}

\medskip

\begin{figure}
    \centering
    \includegraphics[scale=.2]{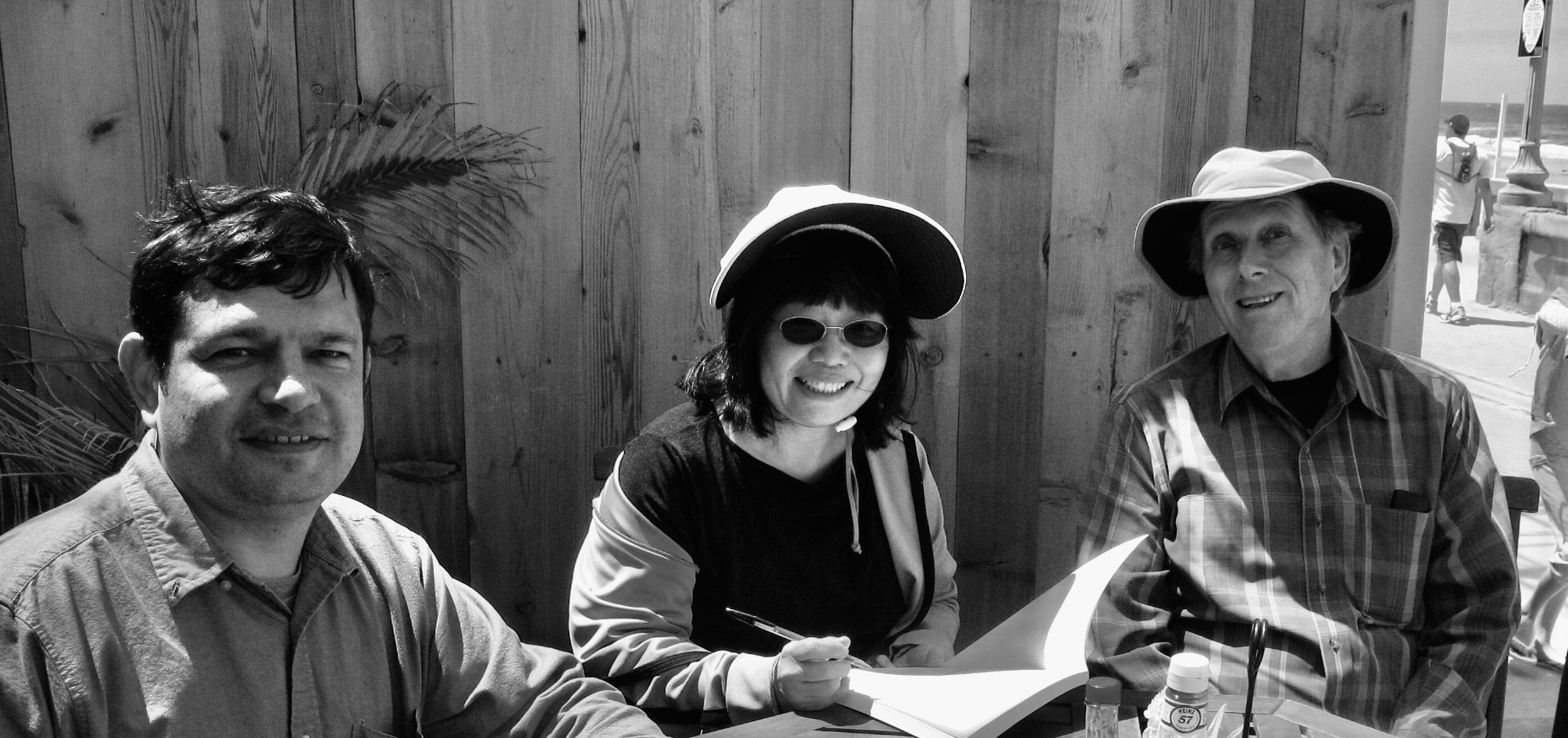}
    \caption*{Ron Graham, Fan Chung and Jozsef Solymosi}
    \label{fig:my_label}
\end{figure}

\section{Square Saturated point sets}
For technical reasons here and in future sections we often switch between integer grids, $[n]\times[n]$ and planes over finite fields, $\mathbb{F}_p\times\mathbb{F}_p.$ 

The next definition we are going to use originates in graph theory. It goes back to a paper from 1964 by Erd\H{o}s, Hajnal and Moon \cite{EHM}. 

\begin{definition}
Given a graph $H,$ a graph $G$ is $H$-saturated if $G$ does not contain $H$ but the addition of an edge joining any pair of nonadjacent vertices of $G$ completes a copy of $H.$ The saturation number of $H,$ written $sat(n,H)$ is the minimum number of edges in an $H$-saturated graph with $n$ vertices (assuming $n\geq|V(H)|$).
\end{definition}

Similar definitions can be given for various combinatorial structures. Here we are going to use the definition for point sets in a plane. The point sets in the definition are subsets of a larger set, a {\em universe} $U,$ like an integer grid $[n]\times[n],$ or a plane over the finite field $\mathbb{F}_p.$
Problems of asking the saturation number for certain subsets of the integer grid $[n]\times[n],$ can be find as early as a paper of Erd\H{o}s and Guy \cite{EG} from 1970, but similar problems were probably considered earlier.

\begin{definition}
Given a point set  $Q,$ another point set $P$ is $Q$-saturated if $P$ does not contain $Q$ but the addition of any point outside of $P$ completes a similar copy of $Q.$ The saturation number of $Q,$ written $sat(U,Q)$ is the minimum number of points in a $Q$-saturated point set in $U.$
\end{definition}

Similarity here means that $Q$ is similar to $Q'$ if there is a transformation $T,$ given by translation rotation and scaling, such that $T(Q)=Q'.$ 

Let us denote the corner, three points with coordinates $(0,0),(1,0),(0,1),$ by $C,$ and the square, four points with coordinates $(0,0),(1,0),(0,1),(1,1),$ by  $Q.$

\begin{claim}
We have the following bounds on the saturation number for sets in $\mathbb{F}_p\times\mathbb{F}_p$ without (tilted) corners
\[
\frac{p}{\sqrt{3}} \leq sat(\mathbb{F}_p\times\mathbb{F}_p,C) \leq p.
\]
\end{claim}
\begin{proof}
Let $S$ be a corner saturated set.  Two elements of $S$ are vertices of three distinct squares, so there are six points which could form a corner with the two elements. There are $p^2$ elements of $\mathbb{F}_p\times\mathbb{F}_p,$ so 
\[
p^2-|S|\leq 6\binom{|S|}{2}
\]
providing the lower bound. The upper bound is a simple construction. Set $$S=\{(0,i) : i\in\mathbb{F}_p\}.$$
Any point outside $S$ with coordinates $(a,b)$ would form a corner with $(0,b),(0,a+b)\in S.$ (Also with $(0,b),(0,b-a)\in S.$)
\end{proof} \qed

\medskip
Both bounds hold in $[n]\times [n]$ as well. It would be interesting to find the sharp bound, or even just a construction in $\mathbb{F}_p\times\mathbb{F}_p$ where $|S|\leq p-1.$

\medskip 
Before stating our next result, we recall a nice result of Katz and Tao \cite{KT} which will be our main tool bounding $sat(U,Q).$ It gives a nontrivial bound on a basic quantity in additive combinatorics.

\medskip

\begin{theorem}[Katz-Tao \cite{KT}]\label{K-T}
Let $A, B,$ be finite subsets of a torsion-free abelian group, and let 
\[
G\subset A\times B \quad \text{be such that}\quad |A|,|B|,|\{a+b:(a,b)\in G\}|\leq N.
\]
Then $|\{a-b:(a,b)\in G\}|\leq N^{11/6}.$
\end{theorem}
The $11/6=1.833\ldots$ exponent is not known to be sharp. A lower bound follows from a variant of a construction of Ruzsa \cite{Ruzsa} showing that the difference set can be as large as $N^{\log(6)/ \log(3)} = N^{1.63093\ldots}.$

\begin{theorem}
Let $p$ be a prime $p\equiv 3 \pmod 4.$ Then
$sat(\mathbb{F}_p\times\mathbb{F}_p,Q)\geq p^{12/11}-p^{3/5},$ i.e. every set which is square-saturated in $\mathbb{F}_p\times\mathbb{F}_p$ has size much larger than the obvious bound, $p.$
\end{theorem}

\begin{proof}
    In this case we can write the elements of $\mathbb{F}_p\times\mathbb{F}_p$ similar to Gaussian integers. We can work on the field $F=\{a+ib : a,b\in\mathbb{F}_p\}.$ Multiplying by $i$ is a rotation by $90$ degrees, so tilted corners are given by $\alpha,\beta,\gamma$ triples where
    
    \[\alpha=(a+ib),\beta=(c+id), \gamma=\alpha+i(\alpha-\beta).\]
 The key observation is that 
 \[
 \alpha=\frac{1+i}{2}\beta+\frac{1-i}{2}\gamma\quad \text{and} \quad -i\left(\frac{1+i}{2}\beta-\frac{1-i}{2}\gamma\right)=\beta+i(\alpha-\beta)
 \]
which is the fourth point of the square determined by $\alpha, \beta, \gamma.$  If $S$ is $Q$-saturated then every point outside $S$ is the fourth point of a square with the other three points in $S.$ We know that $|S|=o(p^2)$ from Theorem \ref{Corner} but here we can simply assume that $|S|\leq p^{12/11}$ otherwise we are done.
We have at least $p^2-p^{12/11}$ points outside of $S,$ all are fourth corners of a square with 3 vertices in $S.$ Let us define a graph $G$ with vertex set $S$ and two elements $(\beta,\gamma)$ are edges iff they are diagonals of a corner. 
Let us consider the sets $A=(1+i)S, B=(1-i)S,$ and a graph $G',$ defined on $A\times B$ as $(a,b)\in G'$ if and only if $(a/(1+i),b/(1-i))\in G.$
With these definitions we have $\{a+b : (a,b)\in G'\}\subset 2S,$ and $|\{a-b : (a,b)\in G'\}|\geq p^2-p^{12/11}.$ We can apply Theorem \ref{K-T} with $N=|S|,$ so 
$p^2-p^{12/11}\leq |S|^{11/6}$ giving the desired bound.
\end{proof}
\qed

\medskip

Note that we didn't use that $S$ was square free, all we used is that any point outside of $S$ would form a square with a corner in $S.$ The very same proof works for $[n]\times[n]$ using Gaussian integers.

\begin{theorem}
    If $S\subset [n]\times[n]$ has the property that for any $a\in \left\{ [n]\times[n]\setminus S \right\}$ there are three elements in $S,$ which form a square with $a,$ then $|S|\geq n^{12/11}+o({n}).$
\end{theorem}

\section{Maximum Corner-free Sets}\label{dens}


In the previous section we gave a bound on the smallest maximal corner free set, here we are going investigate what the size of the maximum corner free set is. This part is not self-contained. We collected references to techniques and analogous results which can be used to tackle our problem. To follow the arguments here, one should be familiar with Fourier methods used to deal with three- and four-term arithmetic progressions up to the level of use of Gowers norms. 
It was pointed out by the anonymous referee that Theorem \ref{t:ERT_upper} below was obtained in nice paper of Prendiville \cite[Corollary 1.3]{Prendiville} and improved in \cite[Theorem 2.21]{Bloom_PhD} by Bloom. 
Our proof below is similar to their work. The main goal is to give a better simple upper bound (on the density of sets without tilted corners) than what is known for axis parallel corners \cite{s_IAN}.

\begin{theorem}\label{Corner}
	Let $A \subseteq [n]^2$ be a set having no isosceles right triangles. Then $|A| = O(n^2/\log^{c_1} n)$. 
	Now if $A$ does not contain squares, then $|A| =  O(n^2/(\log \log n)^{c_2})$, where $c_1,c_2>0$ are some absolute constants. 
\label{t:ERT_upper} 
\end{theorem}

In order to prove the theorem we will see a more general statement, which shows that the estimates for Szemer\'edi's theorem on $k$-term arithmetic progressions can be extended to $k$-element point sets in dimension two. As we mentioned earlier, this part of the paper is not self-explanatory, the statements are heavily dependent on the contents of the cited papers.  

\begin{lemma}\label{lemma}
	Let $k\ge 2$ be a positive integer and  $M_1,\dots,M_{k}$ be $2\times 2$ invertible matrices, $M_i\neq M_j$, $i\neq j$. 
	Also, let $A\subseteq [n]^2$ be a set having no configurations $x,x+M_1 y, \dots, x+M_k y$. 
	Then there is $c_k>0$ such that 
\begin{equation}\label{f:M_j_1}
	|A| = O \left( \frac{n^2}{(\log \log n)^{c_k}} \right) \,,
\end{equation}
	and for $k=2$ there exists $c>0$ with
\begin{equation}\label{f:M_j_2}
	|A| = O \left( \frac{n^2}{(\log n)^{c}} \right) \,. 
\end{equation}
\label{l:M_j}
\end{lemma} 
\begin{proof}
	Consider the quantity 
\[
	\sigma = \sum_{\vec{x},\vec{y}} A(\vec{x}) A(\vec{x}+M_1 \vec{y}) \dots A(\vec{x}+M_k \vec{y}) \,.
\]	
	Let us follow \cite[Proposition 5.3]{GT_primes} in the process  changing  the variables : $\vec{x} = \vec{z}_1+\dots+\vec{z}_k$, $\vec{x} + M_i \vec{y} = \sum_{j=1}^k (I - M_i M_j^{-1}) \vec{z}_j$,
	so $\vec{y} = -\sum_{j=1}^k M_j^{-1} \vec{z}_j$.  
	Since $M_i\neq M_j$, $i\neq j$ it follows that this is  a uniform cover. \footnote{ If
$A$ and $B$ are finite non-empty sets and  $\Phi : A \longrightarrow B$ is a map, then we say that $\Phi$
is a uniform cover of $B$ by $A$ if $\Phi$ is surjective and all the fibers $\{\Phi^{-1}(b) : b \in B\}$ have
the same cardinality.}
	Then $\sigma$ is expressed as 
\[
	\sigma = n^{-k+2} \sum_{\vec{z}_1, \dots, \vec{z}_k} A(\vec{z}_1+\dots+\vec{z}_k) \prod_{j=1}^k f_j (\vec{z}_1, \dots, \vec{z}_k) \,,
\]
	where the function $f_j$, $j\in [k]$ does not depend on $\vec{z}_j$. 
	Hence by the characteristic property of Gowers norms, we see that $\sigma$ is controlled by $U^{k}$--uniformity norm of $A$, see \cite{Gow_m}.  
	Notice also, that the quantity $\sigma$ is affine--invariant. 
	Applying the method from Bourgain's classical paper \cite{Bourgain_AP3_0.5} (or for a sharper bound follow \cite{BS_AP3}) for $k=2,$ and for $k>2$ following the steps of \cite{Gow_4}, \cite{Gow_m}, and \cite{GTZ_inverse}  we obtain a similar 
	bound as in the case of arithmetic progressions of length $k$. 
\end{proof} 
\qed

\bigskip

\noindent
Now we are ready to prove Theorem \ref{Corner} as an easy corollary of Lemma \ref{lemma}.

\medskip

\begin{proof} ({\em of Theorem \ref{Corner}})
	To calculate the number of isosceles right triangles we need to consider
\[
	\sum_{\vec{x},\vec{y}} A(\vec{x}) A(\vec{x}+\vec{y}) A(\vec{x}+\vec{y}^\perp) \,,
\]
	where $\vec{y}=(y_1,y_2)$ and $\vec{y}^\perp = (-y_2,y_1)$. 
	So, in terms of Lemma \ref{l:M_j}, we have

	\[
M_1=\begin{pmatrix}
    1  &  0      \\
    0  &  1      
\end{pmatrix}
\quad
M_2=\begin{pmatrix}
    0  &  -1      \\
    1  &  0      
\end{pmatrix} 
\]
	hence both matrices are invertible.
	In the case of squares  the correspondent quantity is  
\[
	\sum_{\vec{x},\vec{y}} A(\vec{x}) A(\vec{x}+(y_1,y_2)) A(\vec{x}+(-y_2,y_1)) A(\vec{x}+(y_1-y_2,y_1+y_2))  \,,
\]
	and hence 
		\[
M_3=\begin{pmatrix}
    1  &  -1      \\
    1  &  1      
\end{pmatrix}
\]
is invertible as well so we can apply Lemma \ref{lemma}.
\end{proof}
\qed

\begin{remark}
	As we have seen the case of squares corresponds to arithmetic progressions of length four and in this particular case the result can be improved further following the work of Green and Tao in \cite{GT_AP4}.
	Also, it will be interesting to improve Bloom's bound (see \cite{Bloom_PhD})  $|A|=O(n^2/\log^{1-o(1)} n)$ for the maximal size of $A$ having no isosceles right triangles to $|A|=O(n^2/\log^{1+c} n)$, $c>0$, using methods from \cite{BS_AP3}.
\end{remark}

\section{Colouring Problems}\label{col}
In this section we show two results from Euclidean Ramsey theory related to corners. These results follow almost directly from a more general result of the first author's paper {\em ``On some problems of Euclidean Ramsey theory''} \cite{s_Ramsey}. As in the previous section, we are not going to include the details, however we give enough references so that with the cited paper the full proof can be recovered. As we stated in Theorem \ref{GrSo}, colouring the integer grids with few colours results a monochromatic axis parallel corner. Using two colours and relaxing the axis parallel condition will give many monochromatic corners. The systematic investigation of monochromatic triangles in two--colouring of $\mathbb{E}^2$ started in the third paper of the fundamental sequence of papers titled {\em ``Euclidean Ramsey Theorems I. II. III.''} \cite{EGMRSS}.
The next result \cite[Corollary 6]{s_Ramsey} shows that  two--colouring of $\F_p \times \F_p$ always gives as many monochromatic corners as one would expect.

\begin{theorem}[Shkredov \cite{s_Ramsey}]
    Let $p$ be a sufficiently large prime number.
    Then for any two--coloring of the plane $\F \times \F_p$ and any $a,b \neq 0$
    such that
    $a/b$ is a quadratic residue
    there is a monochromatic
    collinear triple $\{ x, y, z\}$
    with
    $\| y-x\| = a$, $\| z-y\| = b$.
\end{theorem}

Actually, by the arguments of the proof of \cite[Theorem 4]{s_Ramsey} we consider $\sigma (R,R,R)$, $\sigma(B,B,B)$ the number of ERT at each colour $R\bigsqcup B = \F_p \times \F_p$ and obtain 
\[
    \sigma (R,R,R) + \sigma (B,B,B)
    = 
\]
\begin{equation}\label{f:sigma}
    = p^{-3}(|R|^3 + |B|^3) + 
    \sigma (R,R,R) + \sigma (B,B,B) +
    3\sigma (\d_R,f_R,f_R) + 3\sigma (\d_B,f_B,f_B) \,,   
\end{equation}
where $f_R (x) = R(x) - |R|/p^2$,
$f_B(x) = B(x) - |B|/p^2$ are the balanced functions of the colours $B$ and $R$, correspondingly. 
As was showed in \cite[Theorem 4]{s_Ramsey} the terms $\sigma (\d_R,f_R,f_R)$, $\sigma (\d_B,f_B,f_B)$ in \eqref{f:sigma} are negligible thanks to the bound for the Kloosterman sums and hence 
\[
    \sigma (R,R,R) + \sigma (B,B,B) 
    = p^{-3} (|R|^3 + |B|^3) + O(p^{5/2}) 
    \ge p^3/4 - C p^{5/2} \,,  
\]
where $C>0$ is an absolute constant. 
As a consequence we obtain 	

\begin{theorem}
	Let $p$ be a prime number. 
	Then for any two--colouring of $\F_p \times \F_p$ the number of monochromatic isosceles right triangles is at least 
\[
	\frac{p^3}{4} + O(p^{5/2}) .
\]
	
\end{theorem} 

A similar argument (now the Kloosterman sums are replaced to the bounds for the zeroth Bessel function) gives 

\begin{theorem}
    Suppose that we have a measurable colouring of the euclidean plane with two colours.
	Then the measure of  monochromatic isosceles right triangles in any of such colouring  is at least 
$
	0.0079 \,.
$
\label{t:mono}
\end{theorem}	

\begin{proof}	
	To obtain the  statement we use \cite[Theorem 10]{s_Ramsey} and derive that the desired measure is at least 
\begin{equation}\label{tmp:28.12_1}
	\frac{1}{4} + \frac{1}{4} \cdot \min_{t\ge 0} \left( 2J_0 (t) + J_0 (\sqrt{2}t) \right) \,, 
\end{equation}
	where $J_0$ is the zeroth Bessel function.
	Using Maple we see that the minimum in \eqref{tmp:28.12_1} is at least $-.9683275949$.
	This completes the proof. 
\end{proof}
\qed

\section{Acknowledgements}
Research of the first author received financial support from the Ministry of Educational and Science of the Russian Federation in the framework of MegaGrant no 075-15-2019-1926. 
Research of the second author was supported in part by an NSERC Discovery grant, OTKA K 119528 and NKFI KKP 133819 grants.
The authors are thankful to the referee for the useful comments and for pointing to important references.

\end{document}